\documentclass{amsart}
\usepackage{amssymb,amsmath,amsrefs}
\usepackage[colorlinks=true]{hyperref}
\hypersetup{urlcolor=blue, citecolor=green}
\usepackage[dvips]{graphicx}
\usepackage{color}

\usepackage{mathrsfs}
\usepackage{tikz}
\usetikzlibrary{arrows.meta, positioning}

\theoremstyle{plain}

\newtheorem{mainthm}{Theorem}

\newtheorem*{conj*}{Conjecture}
\newtheorem*{cor*}{Corollary}
\newtheorem{theorem}{Theorem}[section]

\newtheorem{proposition}[theorem]{Proposition}
\newtheorem{corollary}[theorem]{Corollary}
\newtheorem{lemma}[theorem]{Lemma}

\newtheorem{question}{Question}

\theoremstyle{definition}
\newtheorem*{def*}{Definition}

\newtheorem{example}[theorem]{Example}

\newtheorem{definition}[theorem]{Definition}

\newcommand{\eps}{\varepsilon}
\renewcommand{\epsilon}{\varepsilon}

\newcommand{\Z}{\mathbb{Z}}
\newcommand{\N}{\mathbb{N}}

\newcommand{\dist}{\operatorname{\textit{d}}}

\title[Gluing-orbit and local stable/unstable sets]{Gluing-orbit property, local stable/unstable sets, and induced dynamics on hyperspace.}

\author[Mayara Antunes, Bernardo Carvalho, Welington Cordeiro, José Cueto]{Mayara Antunes, Bernardo Carvalho, Welington Cordeiro, José Cueto}
\date{\today}



\begin{document}
\begin{abstract}
We prove that local stable/unstable sets of homeomorphisms of an infinite compact metric space satisfying the gluing-orbit property always contain compact and perfect subsets of the space. As a consequence, we prove that if a positively countably expansive homeomorphism satisfies the gluing-orbit property, then the space is a single periodic orbit. We also prove that there are homeomorphisms with gluing-orbit such that its induced homeomorphism on the hyperspace of compact subsets does not have gluing-orbit, contrasting with the case of the shadowing and specification properties, proving that if the induced map has gluing-orbit, then the base map has gluing-orbit and is topologically mixing.
\end{abstract}

\maketitle

\section{Introduction and statement of results}


This article aims to study the structure of local stable/unstable sets of homeomorphisms with the gluing-orbit property and to characterize this property for induced maps on the hyperspace of compact subsets. The structure of local stable/unstable sets of a system is central to the study of chaotic dynamics. In hyperbolic systems, the existence of local stable/unstable manifolds is an important feature from where a big part of their dynamics can be understood, see \cite{hirsch1977manifold} for information on the Stable Manifold Theorem.
The idea of using expansiveness and its generalizations to obtain non-trivial local stable/unstable sets comes back to works of Hiraide/Lewowicz of the classification of expansive homeomorphisms on surfaces \cite{hiraide1990expansive,lewowicz1989expansive}, Ma\~n\`e \cite{mane1979expansive}, and Kato \cite{kato1, kato2} in the case of continuum-wise expansive homeomorphisms. More recently, sensitivity and shadowing were used to the same objective, instead of expansiveness \cite{antunes2023sensitivity} (see Theorem \ref{teoremaACT}). It is proved that sensitive homeomorphisms with shadowing have compact and perfect subsets within each local stable/unstable set. Our first main result generalizes this theorem assuming the gluing-orbit property instead of sensitivity and shadowing.

Gluing-orbit was introduced by Bomfim and Varandas in \cite{bomfim2019gluing} and further explored in \cite{bomfim2021gluing} and \cite{sun2020zero} (see Definition \ref{gluingorbit} for the precise definition). It can be seen, in one hand, as a weaker form of specification property that implies transitivity but not necessarily mixing, and in the other hand, for transitive systems it can also be seen as a form of pseudo-orbit tracing property that is weaker than the shadowing property. Indeed, transitive homeomorphisms with shadowing are among the examples of maps satisfying the gluing-orbit property, though there are examples without shadowing such as the irrational rotation on the circle and more generally the minimal equicontinuous maps \cite{sun2020zero}. We use $W^s_\eps(x)$ and $W^u_{\eps}(x)$ to denote the $\eps-$stable and $\eps-$unstable sets of $x$ (see Section 2 for definitions). The following is our first main result.

\begin{mainthm}\label{B} If $f\colon X\to X$ is a homeomorphism of an infinite compact metric space $X$ satisfying the gluing-orbit property, then there is $\eps_0>0$ such that for each $\eps\in(0,\eps_0)$ and each $x\in X$ there are compact and perfect sets $$C_x\subset W^s_\varepsilon(x)\;\;\mbox{ and }\;\;D_x\subset W^u_\varepsilon(x).$$ 
\end{mainthm}

In the proof, we obtain a generalization of the classical Auslander-Yorke dichotomy \cite{ay} which says that a minimal map is either equicontinuous or sensitive. We prove a similar result using gluing-orbit instead of minimality (see Theorem \ref{dicotomia}). Theorem \ref{B} proves, in particular, that $f$ and $f^{-1}$ are not positively countably-expansive, which by definition means that local stable sets are countable, see definition \ref{def_countably_exp}. 
The literature about positively cw-expansive homeomorphisms is growing. The first classical result about this topic is that positive expansiveness for homeomorphisms on compact metric spaces implies the space is finite \cite{keynes1969posit_exp}. 
However, the restriction of the classical example of Denjoy circle diffeomorphism to its minimal invariant cantor subset is a positively 2-expansive homeomorphism \cite{morales2012exp}.  
More examples of positively n-expansive homeomorphisms on Cantor sets are introduced in \cite{li2015levexp} 
for each $n\in\N$ and were further explored in \cite{oprocha2023various}. 
It is known that these examples do not satisfy the shadowing property. Indeed, it was proved in \cite{carvalho2019positively} and \cite{carvalho2020lshadowing} 
that if a positively n-expansive homeomorphism satisfies the shadowing property, then the space is finite, and in \cite{antunes2023sensitivity} 
that if a positively countably expansive homeomorphism is transitive and satisfies the shadowing property, then the space is countable. A consequence of Theorem \ref{B} is that all these examples do not satisfy the gluing-orbit property. This also contrasts with the identity map on a countable and compact set that is positively countably expansive and has shadowing, though it is not transitive.

In the second main result of this article, we discuss the gluing-orbit property for the induced map on the hyperspace of compact subsets. The dynamics of the induced map can be seen as collective dynamics and its chaotic properties can be seen as set-valued chaos, while the base map is seen as the individual dynamics and chaos. Precisely, let $f\colon X\rightarrow X$ be a homeomorphism defined in a compact metric space $X$, $2^X$ denote the space of compact subsets of $X$ endowed with the Hausdorff metric, and $2^f$ be its induced map on $2^X$ defined by $2^f(A)=f(A)$. The study of relations between the dynamics of $f$ and $2^f$ is classical in the literature and many results regarding whether certain dynamical properties hold simultaneously for both maps (or not) are proved. Of particular interest are the results of Bauern and Sigmund \cite{Bauer1975specf} 
that $2^f$ has the specification property if, and only if, $f$ has the specification property, and more recently, Fernandez and Good \cite{fernadez2016shad} 
proved the same is true for the shadowing property. In our second main result we obtain necessary conditions that $f$ has to satisfy so that $2^f$ can have the gluing-orbit property.
 
\begin{mainthm}\label{gluinghiper}
If the induced map $2^f$ has the gluing-orbit property, then $f$ has the gluing-orbit property and is topologically mixing.
\end{mainthm}

A direct and surprising consequence is that there are examples of homeomorphisms with gluing-orbit such that its induced map $2^f$ does not have gluing-orbit. This happens for any transitive but not topologically mixing homeomorphism with gluing-orbit (see Corollary \ref{Cornot} and Example \ref{Example}) and contrasts with the case of shadowing and specification. This result is also related to the following question:

\begin{question}
Do topologically mixing homeomorphisms with gluing-orbit also satisfy the specification property?    
\end{question}

A similar question was stated as Question 1 in \cite{bomfim2019gluing} in the context of $C^1$ diffeomorphisms. We were not able to answer this question, but if it is true, then gluing-orbit and specification would be equivalent for the induced map $2^f$ since specification is equivalent for $f$ and $2^f$. Thus, any $f$ such that $2^f$ has gluing-orbit but not specification would be an example of a topologically mixing homeomorphism with gluing-orbit and without specification.

\section{Gluing-orbit and local stable/unstable sets}


In this section, we state all necessary definitions and prove Theorem \ref{B}. Let $f:X\rightarrow X$ be a homeomorphism defined on a compact metric space $(X,d)$. 

\begin{definition}[Local stable/unstable sets]
We consider the \emph{c-stable set} of $x\in X$ as the set 
$$W^s_{c}(x):=\{y\in X; \,\, d(f^k(y),f^k(x))\leq c \,\,\,\, \textrm{for every} \,\,\,\, k\geq 0\}$$
and the \emph{c-unstable set} of $x$ as the set 
$$W^u_{c}(x):=\{y\in X; \,\, d(f^k(y),f^k(x))\leq c \,\,\,\, \textrm{for every} \,\,\,\, k\leq 0\}.$$
\end{definition}


\begin{definition}[Sensitivity]
A map $f\colon X\rightarrow X$ is \emph{sensitive} if there exists $\varepsilon>0$ such that for every $x\in X$ and every $\delta>0$ there exist $y\in X$ with $d(x,y)<\delta$ and $n\in\mathbb{N}$ satisfying $d(f^{n}(x),f^{n}(y))>\varepsilon$, that is, $B(x,\delta)$ is not contained in $W^s_{\eps}(x)$, where $B(x,\delta)=\{y\in X; \,\,\, d(y,x)<\delta\}$ is the ball centered at $x$ and radius $\delta$. The number $\eps$ is called the \emph{sensitivity constant} of $f$. 
\end{definition}

\begin{definition}[Shadowing]
We say that a homeomorphism $f:X\rightarrow X$ satisfies the \emph{shadowing property} if given $\varepsilon>0$ there is $\delta>0$ such that for each sequence $(x_n)_{n\in\mathbb{Z}}\subset X$ satisfying
$$d(f(x_n),x_{n+1})<\delta \,\,\,\,\,\, \text{for every} \,\,\,\,\,\, n\in\mathbb{Z}$$ there is $y\in X$ such that
$$d(f^n(y),x_n)<\varepsilon \,\,\,\,\,\, \text{for every} \,\,\,\,\,\, n\in\mathbb{Z}.$$
In this case, we say that $(x_k)_{k\in\mathbb{Z}}$ is a $\delta-$pseudo orbit of $f$ and that $(x_n)_{n\in\mathbb{Z}}$ is
$\varepsilon-$shadowed by $y$.
\end{definition}


The following result, proved in \cite{antunes2023sensitivity}, uses sensitivity and shadowing to construct compact and perfect subsets within each local stable/unstable set.

\begin{theorem}\cite{antunes2023sensitivity}\label{teoremaACT} Let $f\colon X\to X$ be a homeomorphism of a compact metric space $X$ satisfying the shadowing property. 
	\begin{enumerate}
		\item If $f$ is sensitive, with sensitivity constant $\eps>0$, then for each $x\in X$ there is a compact and perfect set $$C_x\subset W^u_\varepsilon(x).$$ 
		\item If $f^{-1}$ is sensitive, with sensitivity constant $\eps>0$, then for each $x\in X$ there is a compact and perfect set $$C_x\subset W^s_\varepsilon(x).$$
	\end{enumerate}
\end{theorem}

Consequences to the study of positively countably expansive homeomorphisms were obtained from this result in \cite{antunes2023sensitivity}.

\begin{definition}[Positively countably expansive]\label{def_countably_exp}
We say that a homeomorphism $f:X\rightarrow X$ is \emph{positively countably-expansive} if there exists $\varepsilon>0$ such that $W^s_\varepsilon(x)$ is countable for every $x\in X$. 
\end{definition}

\begin{theorem}\cite{antunes2023sensitivity}\label{ACT}
	Let $f:X\rightarrow X$ be a positively countably-expansive homeomorphism, defined in a compact metric space $X$.
	If at least one of the following conditions is satisfied
	\begin{itemize}
		\item[(1)] $f$ is transitive and has the shadowing property
		\item[(2)] $f$ has the L-shadowing property
	\end{itemize} then $X$ is countable.
\end{theorem}

The L-shadowing property was introduced in \cite{carvalho2019positively} 
and further explored in \cite{carvalho2020lshadowing}. 
In Theorem \ref{B} we generalize these results using the gluing-orbit property. Now we proceed to the definition of the gluing-orbit property following the definition in \cite{sun2020zero}. 

\begin{definition}[Gluing-orbit property]\label{gluingorbit}
For $M\in\N$, we denote by
\[\Sigma_M:=\{1,2,\ldots,M\}^{\N}\] the space of sequences with entries in $\{1,2,\ldots,M\}$.
We call a sequence 
\[\mathscr{C}=\{(x_j,m_j)\}_{j\in\mathbb{N}}\] of ordered pairs in $X\times \mathbb{N}$ an orbit sequence. A gap for an orbit sequence is a sequence
\[\mathscr{G}=\{t_j\}_{j\in\mathbb{N}}\] of positive integers. For a metric space $(X,d)$, a map $f\colon X\rightarrow X$ and $\eps>0$, we say that $(\mathscr{C},\mathscr{G})$ is $\eps$-traced by $z\in X$ if
\[d(f^{s_j+l}(z),f^l(x_j))\leq \eps\;\;\mbox{ for each }\;\;l=0,1,\ldots m_j \;\;\mbox{ and each }\;\; j\in\N,\]
where
\[s_1:=0\;\;\mbox{ and }\;\;s_j:=\sum_{i=1}^{j-1}(m_i+t_i)\;\;\mbox{ for }\;\;j\geq 2.\]
We say that $f$ has the \emph{gluing-orbit property} if for each $\eps>0$ there is $M=M(\eps)>0$ such that for each orbit sequence $\mathscr{C}$, there is a gap $\mathscr{G}\in\sum_M$ such that $(\mathscr{C},\mathscr{G})$ can be $\eps$-traced.
\end{definition}


The following propositions are basic consequences of the gluing-orbit property that will be important to the proof of Theorem \ref{B}.

\begin{proposition}\label{espaçoqueadmitegluing}
    If $f\colon X\rightarrow X$ has the gluing-orbit property, then either $X$ is perfect, or it consists of a single periodic orbit.
\end{proposition}
\begin{proof}
If $X$ is not perfect, there is $x\in X$ an isolated point. Choose $\eps>0$ such that $B(x,\eps)\cap X=\{x\}$ and let $M=M(\eps)$ be given by the gluing-orbit property. Then there is $z\in X$ which $\eps$-traces the orbit sequence $\mathscr{C}=\{(x,0),(x,0)\}$ with some gap $\mathscr{G}=\{g\}\subset\{1,\ldots,M\}$. Thus, $d(x,z)<\eps$ and $d(f^{g}(z),x)<\eps$, which imply that $z=x=f^g(z)$ since $B(x,\eps)$ is an isolating neighborhood. Therefore, $x$ is a periodic point with period at most $M$. Suppose that there is $y\in X$ that it is not in the orbit of $x$ and choose 
$$\delta<\min\{\eps,d(x,y),d(f(x),y),\ldots, d(f^{g-1}(x),y)\}.$$ The gluing-orbit property ensures the existence of $N=N(\delta)$ and $z\in X$ that $\delta$-traces the orbit sequence $\mathscr{C}=\{(x,0), (y,0)\}$ with a gap $\mathscr{G}'=\{g'\}\in \{1,2,\ldots, N\}$. Again, $x=z$ since $B(x,\delta)\subset B(x,\eps)=\{x\}$, which implies that $g'\in\{1,\dots,g-1\}$, and $d(f^{g'}(x),y)<\delta$, which contradicts the choice of $\delta$. This proves that if $X$ contains an isolated point, this point is periodic and $X$ consists of its periodic orbit.
\end{proof}


\begin{proposition} \label{gluingorbitinversa}
If $f:X\rightarrow X$ is a homeomorphism defined on a compact metric space, then $f$ has the gluing-orbit property if, and only if, $f^{-1}$ has the gluing-orbit property.
\end{proposition}
\begin{proof}
Assume that $f$ has the gluing-orbit property, let $\eps>0$ be given, and let $M=M(\eps)\in\mathbb{N}$ be given by the gluing-orbit property. The proof goes by noting that if
\[\mathscr{C}=\{(x_j,m_j)\}_{1\leq j\leq k}\] is a finite orbit sequence for $f^{-1}$, and \[\mathscr{C}' = \{(f^{-m_{k+1-j}}(x_{k+1-j}),m_{k+1-j})\}_{1\leq j\leq k}\] is a finite orbit sequence for $f$. Then, if $z\in X$ $\eps$-traces $(\mathscr{C}',\mathscr{G}')$ with some gap $\mathscr{G}'=\{t_j\}_{1\leq j\leq k}$, then $f^{s_k+m_1}(z)$ $\eps$-traces $(\mathscr{C},\mathscr{G})$ for $f^{-1}$, where $\mathscr{G}=\{t_{k-j}\}_{1\leq j\leq k}$ and \[s_1:=0\;\;\;\mbox{ and } \;\;\;s_j:=\sum_{i=1}^{j-1}(m_i+t_i)\;\;\;\mbox{ for }\;\;\;2\leq j\leq k\] (see Figure \ref{figura:tracing}).

\begin{center}
    \begin{figure}[h]
	\centering 
	\includegraphics[width=13cm]{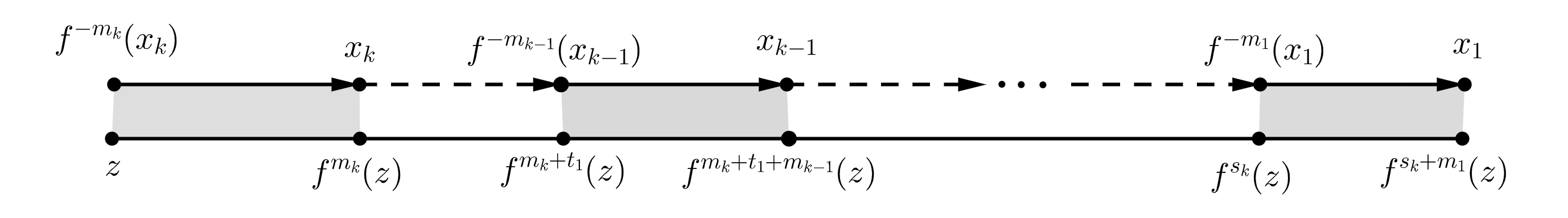} 
		\label{figura:tracing}
	\caption{The tracing of $(\mathscr{C}',\mathscr{G}')$ by $z$.}
\end{figure}
\end{center}

\end{proof}

In \cite{sun2020zero} 
the maps satisfying gluing-orbit and zero topological entropy were characterized as follows.

\begin{theorem}\label{sunequivalencia}\cite[Theorem 1.2]{sun2020zero}
    Let $f:X\rightarrow X$ be a continuous system defined on a compact metric space $X$ with the gluing orbit property. The following are equivalent:
    \begin{enumerate}
        \item $f$ has zero topological entropy.
        \item $f$ is minimal.
        \item $f$ is equicontinuous.
        \item $f$ is uniformly rigid.
    \end{enumerate}
\end{theorem}

We recall all definitions of objects in this result since they will be necessary for the proof of Theorem \ref{B}.

\begin{definition}[Topological entropy]
    Given $n\in\N$ and $\delta>0$, we say that $E\subset X$ is $(n,\delta)$-\emph{separated}
if for each $x,y\in E$, $x\neq y$, there is $k\in \{0,\dots,n\}$ such that
$\dist(f^k(x),f^k(y))>\delta$.
Let $s(n,\delta)$ denotes the maximal cardinality of an $(n,\delta)$-separated subset $E\subset X$ (since $X$ is compact, $s(n,\delta)$ is finite).
Let \[
 h(f,\delta)=\limsup_{n\to\infty}\frac 1n\log s_n(f,\delta).
\]
Note that $h(f,\delta)$ increases as $\delta$ decreases to $0$ and define the \textit{topological entropy} of $f$ by $$h(f)=\lim_{\delta\to 0}h(f,\delta).$$
\end{definition} 



\begin{definition}[Uniformly rigid and equicontinuous maps]
We say that a map $f\colon X\rightarrow X$ in a compact metric space is \emph{uniformly rigid} if there is an increasing sequence $(n_k)_{k\in\N}$ such that $(f^{n_k})_{k\in\N}$ converges uniformly to the identity. It is called \emph{equicontinuous} if the sequence of iterates $(f^n)_{n\in\N}$ is an equicontinuous sequence of maps, that is, for each $\eps>0$ there exists $\delta>0$ such that $$B(x,\delta)\subset W^s_{\eps}(x) \,\,\,\,\,\, \text{for every} \,\,\,\,\,\, x\in X.$$
\end{definition}


\begin{definition}[Transitive point and minimal map]
We say that $p\in X$ is a transitive point if its future orbit is dense in $X$ and denote by $Tran(X,f)$ the set of all transitive points of $f$. We say that $f$ is \emph{minimal} if the future orbit of every point of the space is dense in the space, that is, $X=Tran(X,f)$.
\end{definition}

\begin{lemma}\label{sun_1}\emph{[}Lemma 3.2 in \cite{sun2020zero}\emph{]} If $(X,f)$ is not uniformly rigid, then there is $\gamma>0$ such that for every $p\in Tran(X,f)$ and every $m\in\mathbb{N}$, there is $\tau=\tau(p,m)\in\mathbb{N}$ such that $$d(f^\tau(p),f^\tau(f^m(p)))>\gamma.$$
\end{lemma}

This means that no iterate of $p$ can belong to the $\gamma$-stable set of $p$ and gives a way to ensure the existence of close points with separating future orbits as follows.

\begin{proposition}\label{sensitivity}If $X$ is a compact metric space and $f\colon X\rightarrow X$ is a transitive map that is not uniformly rigid, then $f$ is sensitive.
\end{proposition}

\begin{proof}
We will prove that $\eps=\frac{\gamma}{4}$, where $\gamma>0$ is given by Lemma \ref{sun_1}, is a sensitivity constant of $f$. Let $\delta>0$ and $x\in X$.
If $x\in Tran(X,f)$, then there is $m>0$ such that $y=f^m(x)\in B(x,\delta)$ and Lemma \ref{sun_1} ensures the existence of $\tau\in\mathbb{N}$ such that $$d(f^\tau(x),f^\tau(y))>\gamma>\eps.$$
If $x\notin Tran(X,f)$, we can use the transitivity of $f$ to choose $p\in Tran(X,f)\cap B(x,\delta)$ and consider some neighborhood $U$ of $p$ such that $U\subset B(x,\delta)$. There is $m>0$ such that $y=f^m(p)\in U$ and by Lemma \ref{sun_1} there is $\tau>0$ such that
\begin{equation}\label{eq1}
d(f^\tau(p),f^\tau(y))>\gamma.
\end{equation} 
If $f^\tau(x)\notin B(f^\tau(p),\eps)\cup B(f^\tau(y),\eps)$, then both
$$d(f^\tau(p),f^\tau(x))>\eps \,\,\,\,\,\, \text{and} \,\,\,\,\,\, d(f^\tau(y),f^\tau(x))>\eps.$$
If $f^\tau(x)\in B(f^\tau(p),\eps)\cup B(f^\tau(y),\eps)$, then either $f^\tau(x)\in B(f^\tau(p),\eps)$ and (\ref{eq1}) ensures that $d(f^\tau(x),f^\tau(y))>\eps$, or $f^\tau(x)\in B(f^\tau(y),\eps)$ and (\ref{eq1}) ensures that $d(f^\tau(x),f^\tau(p))>\eps$.
In all cases we found points in $B(x,\delta)$ with future iterates that separate $\eps$ from the respective iterate of $x$. Since this holds for every $\delta>0$ and $x\in X$, sensitivity is proved.
\end{proof}

Using these results, we can conclude sensitivity for every homeomorphism with gluing-orbit and positive topological entropy.

\begin{theorem}\label{dicotomiagluing}
If $f\colon X\rightarrow X$ is a map of a compact metric space $X$ satisfying the gluing-orbit property and having positive topological entropy, then $f$ is sensitive.
\end{theorem}

\begin{proof}
    Theorem \ref{sunequivalencia} and 2.17(3) in \cite{sun2020zero} ensure that gluing-orbit and positive topological entropy imply that $f$ is topologically transitive and not uniformly rigid. Thus, Proposition \ref{sensitivity} ensures that $f$ is sensitive.
\end{proof}

We note that in the case of positive topological entropy, there is not a classification result such as Sun's classification in the zero topological entropy case. Thus, understanding consequences of gluing-orbit in this case is an interesting topic. A direct consequence of what we have proved is the following Auslander-Yorke dichotomy exchanging minimality by the gluing-orbit property.

\begin{theorem}\label{dicotomia}
If $f\colon X\rightarrow X$ is a map of a compact metric space satisfying the gluing-orbit property, then either $f$ is minimal and equicontinuous or $f$ is sensitive.
\end{theorem}

\begin{proof}
If $f$ has zero topological entropy, then it is minimal and equicontinuous (see Theorem \ref{sunequivalencia}), and if $f$ has positive topological entropy, we conclude sensitivity from Theorem \ref{dicotomiagluing}.
\end{proof}

We are ready to prove Theorem \ref{B}.

%

\begin{thebibliography}{99]}

\bibitem{akin1997equicont} E. Akin, Recurrence in topological dynamics: Furstenberg families and ellis actions, \emph{Plenum Press}, (1997).
\bibitem{antunes2023sensitivity} M. Antunes, B. Carvalho, and M. Tacuri, Sensitivity, local stable/unstable sets and shadowing, \emph{Dynamical Systems} \textbf{38} (2023), no. 3, 477–489.
\bibitem{carvalho2020lshadowing} A. Artigue, B. Carvalho, W. Cordeiro, and J. Vietez, Beyond topological hyperbolicity: The L-shadowing property, \emph{Journal of Differential Equations} \textbf{268} (2020), 3057–3080.
\bibitem{ay} J. Auslander and J. A. Yorke, Interval maps, factors of maps, and chaos, \emph{Tohoku Mathematical Journal}, Second Series \textbf{32} (1980), no. 2, 177–188.
\bibitem{Bauer1975specf} W. Bauer and K. Sigmund, Topological dynamics of transformations induced on the space of probability measures, \emph{Monatsh Math} \textbf{79} (1975), 81–92.
\bibitem{bomfim2021gluing} T. Bomfim, M. J. Torres, and P. Varandas, Gluing orbit property and partial hyperbolicity, \emph{Journal of Differential Equations} \textbf{272} (2021), 203–221.
\bibitem{bomfim2019gluing} T. Bomfim and P. Varandas, The gluing orbit property, uniform hyperbolicity and large deviations principles for semiflows, \emph{Journal of Differential Equations} \textbf{267} (2019), no. 1, 228–266.
\bibitem{carvalho2019positively} B. Carvalho and W. Cordeiro, Positively N-expansive homeomorphisms and the L-shadowing property, \emph{Journal of Dynamics and Differential Equations} \textbf{31} (2019), no. 2, 1005–1016.
\bibitem{fernadez2016shad} L. Fernandez and C. Good, Shadowing for induced maps of hyperspaces, \emph{Fundamenta Mathematicae}\textbf{ 235} (2016), 277–286.
\bibitem{oprocha2023various} P. Guillon, S. Gangloff, and P. Oprocha, Various questions around finitely positively expansive dynamical systems, \emph{CRC Press} (2023).
\bibitem{hiraide1990expansive} K. Hiraide, Expansive homeomorphisms of compact surfaces are pseudo-anosov (1990).
\bibitem{hirsch1977manifold} M. Hirsch, C. Pugh, and M. Shub, Invariant manifolds, \emph{Lecture Notes in Math} \textbf{583} (1977).
\bibitem{kato2} H. Kato, Concerning continuum-wise fully expansive homeomorphisms of continua, \emph{Topology and its Applications} \textbf{53} (1993), no. 3, 239–258.
\bibitem{kato1} H. Kato, Continuum-wise expansive homeomorphisms, \emph{Canadian Journal of Mathematics} \textbf{45} (1993), no. 3, 576–598.
\bibitem{keynes1969posit_exp} H. B. Keynes and J. B. Robertson, {Genertors for topological entropy and expansiveness}, \emph{Math. Systems Theory} \textbf{3 }(1969), 51–59.
\bibitem{lewowicz1989expansive} J. Lewowicz, Expansive homeomorphisms of surfaces, \emph{Boletim da Sociedade Brasileira de Matemática-Bulletin/Brazilian Mathematical Society} \textbf{20} (1989), no. 1, 113–133.
\bibitem{li2015levexp} J. Li and R.Zhang, Levels of generalized expansiveness, \emph{Journal of Dynamics and Differential Equations} \textbf{29} (2015), no. 2,
1–18.
\bibitem{douglas1995symbolic} D. Lind and B. Marcus, An introduction to symbolic dynamics and coding, \emph{Cambridge university press}, 1995.
\bibitem{mane1979expansive} R. Mãné, Expansive homeomorphisms and topological dimension, \emph{Transactions of the American Mathematical Society}\textbf{ 252} (1979), 313–319.
\bibitem{morales2012exp} C. A. Morales, A generalization of expansivity, \emph{Discrete and continuous dinamical systems} \textbf{32} (2012), 293–301.
\bibitem{sun2020zero} P. Sun, Zero-entropy dynamical systems with the gluing orbit property, \emph{Advances in Mathematics} \textbf{372} (2020), 107294.
\bibitem{krerley2019entropy} M. Viana and K. Oliveira, Fundamentos da teoria ergódica, SBM, 2019.
\end{thebibliography}

\begin{proof}[Proof of Theorem \ref{B}]
We assume that $f$ has gluing-orbit, so Theorem \ref{dicotomia} ensures that $f$ is either equicontinuous or sensitive. Proposition \ref{espaçoqueadmitegluing} ensures that the space $X$ is perfect since by hypothesis it is infinite. In the case $f$ is equicontinuous, $f^{-1}$ is also equicontinuous \cite{akin1997equicont}. 
Let $\eps_0$ be the diameter of $X$ and for each $\eps\in(0,\eps_0)$ choose $\delta>0$ such that 
$$B(x,\delta)\subset W^s_\eps(x)\cap W^u_\eps(x) \,\,\,\,\,\, \text{for every} \,\,\,\,\,\, x\in X.$$ Since $X$ is a perfect metric space, it follows that $\overline{B(x,\delta)}$ is a compact and perfect set which is contained in both $W^s_\eps(x)$ and $W^u_\eps(x)$. Now
assume that $f$ is sensitive with sensitivity constant $\eps_0>0$ and let $\eps\in(0,\eps_0)$. The gluing-orbit property assures the existence of $M=M(\eps)>0$ such that for every orbit sequence $\mathscr{C}$, there is a gap $\mathscr{G}\in\sum_M$ such that $(\mathscr{C},\mathscr{G})$ can be $\frac{\varepsilon}{2}$-traced. Given $x\in X$, the sensitivity of $f$ ensures the existence of $x_1\in X$ such that 
$$x_1\notin W^s_\varepsilon(f^n(x)) \,\,\,\,\,\, \text{for every} \,\,\,\,\,\, n\in\{1,\ldots M\}.$$ In fact, otherwise, $X=\bigcup_{n=1}^MW_\varepsilon^s(f^n(x))$ and then, since each local stable set is closed, $W^s_\varepsilon(f^n(x))$ would be an open subset of $X$ for every $n\in\{1,\ldots, M\}$, contradicting sensitivity of $f$. For each $l\in \mathbb{N}$, consider the orbit sequence
\[\{(f^{-l}(x),l), (x_1,l)\}.\] The gluing-orbit property ensures that there are $z_l\in X$ and \[\mathscr{G}_l=\{t_l\}\in \{1,2,\ldots, M\}\] such that 
$$d(f^{-n}(x),f^{-n}(z_l))\leq \frac{\varepsilon}{2} \,\,\,\,\,\, \text{and} \,\,\,\,\,\, d(f^{n}(x_1),f^{n+t_l}(z_l))\leq \frac{\varepsilon}{2}$$
for every $n\in\{0,1,\ldots,l\}$. We can choose a subsequence $\{z_{l_j}\}_{j\in\mathbb{N}}\subset \{z_l\}_{l\in\mathbb{N}}$ such that $t_{l_j}=t$ for any $j\in\mathbb{N}$ and that there exists $\displaystyle c_1(x)=\lim_{j\rightarrow \infty}z_{l_j}$. Thus, the following inequalities hold for every $n\in\mathbb{N}$:
\[\begin{array}{rcl}
d(f^{-n}(x),f^{-n}(c_1(x)))&=&\displaystyle d(f^{-n}(x),f^{-n}(\lim_{j\rightarrow \infty}z_{l_j}))\\
&=&\displaystyle\lim_{j\rightarrow \infty}d(f^{-n}(x),f^{-n}(z_{l_j}))\\
&\leq &\displaystyle\frac{\varepsilon}{2}.
\end{array}\] and
\[d(f^n(x_1),f^{n+t}(c_1(x)))<\frac{\varepsilon}{2}.\] These imply that $c_1(x)\in W^u_{\frac{\varepsilon}{2}}(x)$ and $f^t(c_1(x))\in W^s_{\frac{\varepsilon}{2}}(x_1)$.
Note that $c_1(x)\neq x$ since $f^t(c_1(x))\in W^s_{\frac{\varepsilon}{2}}(x_1)$ and $x_1\notin W^s_{\varepsilon}(f^t(x))$, and consider the set $$C_1=\{x, c_1(x)\}.$$
Let $\varepsilon_1>0$ be such that $$\eps_1<\min\left\{\frac{\eps}{4},\frac{d(x,c_1(x))}{2}\right\}$$
and choose $M_1$, given by the gluing-orbit property, such that for every orbit sequence $\mathscr{C}$, there is a gap $\mathscr{G}\in\sum_M$ such that $(\mathscr{C},\mathscr{G})$ can be ${\varepsilon_1}$-traced. As before, we can use the sensitivity of $f$ for each $y\in C_1$ to obtain $y_1(y)$ such that $$y_1(y)\notin W^s_\varepsilon(f^n(y)) \,\,\,\,\,\, \text{for every} \,\,\,\,\,\, n\in\{1,2,\ldots, M_1\}$$ 
and the gluing-orbit property to obtain $t_y\in\{1,\ldots,M\}$ and $c_2(y)$ such that $$c_2(y)\in W^u_{\varepsilon_1}(y) \,\,\,\,\,\, \text{and} \,\,\,\,\,\, f^{t_y}(c_2(y))\in W^s_{\varepsilon_1}(y_1).$$ 
Note that $c_2(y)\in W^u_\varepsilon(x)$ for every $y\in C_1$ since 
$$y\in W^u_{\frac{\varepsilon}{2}}(x), \,\,\,\,\,\, c_2(y)\in W^u_{\varepsilon_1}(y) \,\,\,\,\,\, \text{and} \,\,\,\,\,\, \frac{\eps}{2}+\eps_1<\eps.$$ 
Also, $c_2(y)\neq y$ since $f^{t_y}(c_2(y))\in W^s_{\varepsilon_1}(y_1(y))$ and $y_1(y)\notin W^s_{\varepsilon}(f^{t_y}(y))$. Moreover, 
$$c_2(y)\neq z \,\,\,\,\,\, \text{for each} \,\,\,\,\,\, z\in C_1$$ because $d(c_2(y),y)<\eps_1$ and $d(y,z)>\eps_1$ if $z\in C_1\setminus\{y\}$.
Thus, the set $$C_2=C_1\cup \{ c_2(y); \,\,y\in C_1\}$$ has $2^2$ elements, $C_2\subset W^u_\varepsilon(x)$ and for each $y\in C_{1}$ there is $c_2(y)\in C_2$ such that $$d(c_2(y),y)<\frac{\eps}{2^2}.$$
We can construct using an induction process an increasing sequence of sets $(C_k)_{k\in\mathbb{N}}$ such that $C_k$ has $2^k$ elements, $C_k\subset W^u_\varepsilon(x)$ and for each $y\in C_{k-1}$ there exists $c_k(y)\in C_k$ such that $$d(c_k(y),y)<\frac{\eps}{2^k}.$$
Thus, we can consider the set $$C_x=\overline{\bigcup_{k\geq 1}C_k},$$ that is a compact set contained in $W^u_\varepsilon(x)$, since $W^u_\varepsilon(x)$ is closed and $C_k\subset W^u_\varepsilon(x)$ for every $k\in\N$. To see that $C_x$ is perfect let $z\in C_x$. If $z\notin C_k$ for every $k\in\N$, then clearly $z$ is accumulated by points of $C_x$. If $z\in C_k$ for some $k\in\N$, then 
$$z\in C_n \,\,\,\,\,\, \text{for every} \,\,\,\,\,\, n\geq k,$$ since $(C_k)_{k\in\N}$ is an increasing sequence. Thus, for each $\alpha>0$ we can choose $N>k$ such that $$\frac{\eps}{2^N}<\alpha$$ and since $z\in C_N$ it follows that there exists $c_N(z)\in C_{N+1}$ satisfying
$$d(c_N(z),z)<\frac{\eps}{2^N}<\alpha.$$ So, for each $z\in C_x$ and $\alpha>0$ we can find $c_N(z)\in C_x$ such that $d(z,c_N(z))<\alpha$. This proves that $z$ is an accumulation point of $C_x$ and that $C_x$ is perfect. The construction of the set $D_x\subset W^s_\eps(x)$ is analogous, just notice that $f$ and $f^{-1}$ are either both equicontinuous or both sensitive because of Proposition \ref{gluingorbitinversa}, Theorem \ref{dicotomiagluing} and the fact that $h(f)=h(f^{-1})$ (see Proposition 10.1.12 in \cite{krerley2019entropy}). 
\end{proof}

As a corollary, we generalize Theorem \ref{ACT} exchanging transitivity and shadowing by the gluing-orbit property. Our conclusion is a bit stronger than the space being countable, indeed we prove that the space consists of a single periodic orbit.

\begin{corollary}\label{espaço=orbita}
If $f\colon X\rightarrow X$ is a positively countably expansive homeomorphism, defined on a compact metric space, satisfying the gluing-orbit property, then $X$ consists of a single periodic orbit.
\end{corollary}

\begin{proof}
If $X$ is not a single periodic orbit, Proposition \ref{espaçoqueadmitegluing} ensures that $X$ is perfect and Theorem \ref{B} ensures the existence of perfect sets in the local stable sets, which are uncountable and contradicts the fact that $f$ is positively countably expansive.
\end{proof}



\section{Induced dynamics on the hyperspace}

In this section, we state necessary definitions, prove Theorem \ref{gluinghiper}, and analyze its consequences.

\begin{definition}[Hyperspace map]
Recall that $X$ is a compact metric space and consider the hyperspace of closed nonempty subsets of $X$:
\[2^X=\{A\subset X\;|\;A \mbox{ is nonempty and closed}\}\]
endowed with the Hausdorff distance
$$d_H(A,C)=\inf\{\eps>0; \,\,\, A\subset B(C,\eps) \,\,\, \text{and} \,\,\, C\subset B(A,\eps)\},$$
where $B(A,\eps)=\{x\in X; d(x,A)<\eps\}$ is the $\eps$-neighborhood of $A$.
Let $f\colon X\rightarrow X$ be a homeomorphism and define the \textit{induced hyperspace map} as the map $2^f:2^X\rightarrow 2^X$ given by $2^f(A)=f(A)$.
\end{definition}

\begin{definition}[Transitivity and Mixing]
A map $f\colon X\to X$ is called \emph{transitive}, if for any pair $U,V\subset X$ of non-empty open subsets, there exists $n\in\N$ such that $$f^n(U)\cap V\neq\emptyset.$$ It is called \emph{topologically mixing} if for any pair $U,V\subset X$ of non-empty open subsets, there exists $n\in\N$ such that $$f^k(U)\cap V\neq\emptyset \,\,\,\,\,\, \text{for every} \,\,\,\,\,\, k\geq n.$$
\end{definition}

\begin{proof}[Proof of Theorem \ref{gluinghiper}]
First, we assume that $2^f$ has the gluing-orbit property and prove that $f$ has the gluing-orbit property. Given $\eps>0$, let $M=M(\eps/2)$ be the constant given by the gluing-orbit property for $2^f$. Consider a finite orbit sequence for $f$
\[\mathscr{C}=\{(x_j,m_j)\}_{1\leq j\leq k}\] and note that
\[\mathscr{C}'=\{(\{x_j\},m_j)\}_{1\leq j\leq k}\] is a finite orbit sequence for $2^f$. Then, by the gluing-orbit property of $2^f$, there are $A\in 2^X$ and a gap $\mathscr{G}'=\{t_j\}_{1\leq j\leq k}$ such that $(\mathscr{C}',\mathscr{G}')$ is $\frac{\eps}{2}$-traced by $A$, that is,
\[d_H((2^f)^{s_j+l}(A),(2^f)^{l}(\{x_j\}))=d_H(f^{s_j+l}(A),f^l(\{x_j\}))\leq \frac{\eps}{2},\]
for every $l\in\{0,1,\ldots,m_j\}$ and $j\in\{1\dots,k\}$, where $s_1:=0$ and $\displaystyle s_j:=\sum_{i=1}^{j-1}(m_i+t_i)$. In particular, any $z\in A$ satisfies:
\[d(f^{s_j+l}(z), f^l(x_j))\leq \eps\;\;\;\mbox{for every }\;\;\;l\in\{0,1,\ldots,m_j\},\] that is, $(\mathscr{C},\mathscr{G}')$ is $\eps$-traced by any point of $A$, which proves that $f$ has the gluing-orbit property. Now, we assume that $2^f$ has the gluing-orbit property and prove that $f$ is topologically mixing. Let $U,V$ be non-empty open subsets of $X$ and choose $\varepsilon>0$ and points $x,y\in X$ such that $$B(x,\varepsilon)\subset U \,\,\,\,\,\, \text{and} \,\,\,\,\,\, B(y,\varepsilon)\subset V.$$ Let $M=M(\eps)$ be the constant given by the gluing-orbit property of $2^f$. We will prove that 
$$f^{M+k}(U)\cap V\neq \emptyset \,\,\,\,\,\, \mbox{ for every } \,\,\,\,\,\, k\geq 0.$$ Since this can be done for each pair of non-empty open subsets, it proves that $f$ is topologically mixing. For each $k\geq 0$, consider the set
$$E_k=\{y,f^{-1}(y),\ldots, f^{-M-k}(y)\}$$ and the orbit sequence $\{(\{x\},0),(E_k,M+k)\}$. Since $2^f$ has the gluing-orbit property, for each $k\geq 0$, there exist $A_k\in 2^X$ and a gap $N_k\in\{0,\dots,M\}$ such that $\{(\{x\},0),(E_k,M+k)\}$ is $\eps$-traced by $A_k$ with gap $N_k$, that is, $d_H(A_k,\{x\})<\eps$ and
\[d_H((2^f)^i(E_k), (2^f)^{i+N_k}(A_k)) = d_H(f^i(E_k),f^{i+N_k}(A_k))<\eps\]
for every $i\in\{0,\dots,M+k\}$. Note that 
$$f^i(E_k)=\{f^i(y),f^{i-1}(y),\ldots, f^{i-M-k}(y)\}$$ 
and that 
$$y\in f^i(E_k) \,\,\,\,\,\, \text{for every} \,\,\,\,\,\, i\in\{0,\ldots, M+k\}.$$ Thus, the previous inequalities ensure that for each $i\in\{0,\ldots,M+k\}$ there exists 
$$y_i\in f^{i+N_k}(A_k)\cap B(y,\eps),$$ and since $f^{-(i+N_k)}(y_i)\in A_k\subset B(x,\eps)$, it follows that
$$y_i\in f^{i+N_k}(U)\cap V.$$ In particular, for $i=M+k-N_k$ we have
$$y_{M+k-N_k}\in f^{M+k}(U)\cap V$$ and this finishes the proof. 

\end{proof}




As observed in the introduction, the following corollary is a direct consequence of Theorem \ref{gluinghiper}. 

\begin{corollary}\label{Cornot}
If a transitive map has gluing-orbit but is not topologically mixing, then $2^f$ does not have gluing-orbit.
\end{corollary}

We present below a classical example of a transitive but not mixing shift of finite type that is included in this case (see \cite[Example 3.1]{bomfim2019gluing}).

\begin{example}\label{Example}
    Consider the matrix
    \[M=\left(\begin{array}{cccc}
    0&1&0&0\\
    1&0&1&0\\
    0&1&0&1\\
    0&0&1&0
    \end{array}\right)\]
and let $M(i,j)$ be the element of $M$ which belongs to row $i$ and column $j$. The following direct graph is associated to this matrix by considering the set of vertices $\{1,2,3,4\}$ and an edge is drawn from $i$ to $j$ if $M(i,j)=1$. 
\begin{figure}[h!]
\centering
\begin{tikzpicture}[
    node distance=1cm,
    every node/.style={circle, draw, minimum size=0.5cm},
    every path/.style={->, thick}
]
    \node (1) {1};
    \node (2) [right=of 1] {2};
    \node (3) [right=of 2] {3};
    \node (4) [right=of 3] {4};

    \draw (1) edge[bend left] (2);
    \draw (2) edge[bend left] (1);
    \draw (2) edge[bend left] (3);
    \draw (3) edge[bend left] (2);
    \draw (3) edge[bend left] (4);
    \draw (4) edge[bend left] (3);
\end{tikzpicture}
\caption{Graph associated to the matrix $M$}
\label{fig:diagram}
\end{figure}
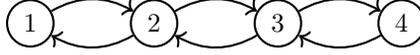

Consider the set of admissible sequences $$\Omega_M=\{(x_k)_{k\in\mathbb{Z}}\in\{1,2,3,4\}^{\Z}\;;\; M(x_k,x_{k+1})=1\}$$ that are obtained as two-sided paths on the graph above. The shift map $\sigma\colon\Omega_M\rightarrow \Omega_M$ defined by $\sigma((x_k)_{k\in\mathbb{Z}})=(x_{k+1})_{k\in\mathbb{Z}}$
is topologically transitive since topological transitivity of $\sigma$ is the same as irreducibility of $M$ (see Example 6.3.2 in \cite{douglas1995symbolic}). Recall that a non-negative matrix $A$ is irreducible if for each pair of indices $(i,j)$, there exists some $n\geq0$ such that $A^n(i,j)>0$. We can easily see that $M$ is irreducible considering the iterates 

\[M^2=\left(\begin{array}{cccc}
    1&0&1&0\\
    0&2&0&1\\
    1&0&2&0\\
    0&1&0&1
    \end{array}\right) \;\;\mbox{ and }\;\;M^3=\left(\begin{array}{cccc}
    0&2&0&1\\
    2&0&3&0\\
    0&3&0&2\\
    1&0&2&0
    \end{array}\right).\]
This ensures that $\sigma$ has the gluing-orbit property by \cite{bomfim2019gluing}*{Example 3.1}. But $\sigma $ is topologically mixing if, and only if, $M$ is irreducible and of period $1$ (see Proposition 4.5.10 in \cite{douglas1995symbolic}). Recall that the period of a state $i\in\{1,2,3,4\}$, denoted by $per(i)$, is the greatest common divisor of those integers $n\geq1$ for which $M^n(i,i)>0$ (if no such integer exist we let $per(i)=\infty$). The period of $M$, denoted by $per(M)$, is the greatest common divisor of the numbers $(per(i))_{i\in\{1,2,3,4\}}$ that are finite. 
Note that $M^n(i,j)$ represents the number of paths on the graph starting at the vertex $i$ and ending at the vertex $j$ in exactly $n$ steps. Thus, for each $i\in\{1,2,3,4\}$, $M^n(i,i)>0$ if, and only if, $n$ is even.
Consequently, $per(i)=2$ for every $i\in\{1,2,3,4\}$ and $per(M)=2$. This proves that $\sigma$ is not topologically mixing and
Theorem \ref{gluinghiper} (Corollary \ref{Cornot}) ensures that $2^{\sigma}$ does not have the gluing-orbit property.
\end{example}

\hspace{-0.4cm}\textbf{Acknowledgments.}
Bernardo Carvalho was supported by Progetto di Eccellenza MatMod@TOV grant number PRIN 2017S35EHN and Mayara Antunes, Welington Cordeiro, and Jos\'e Cueto were also supported by Fapemig grant number APQ-00036-22.

\vspace{1.5cm}
\noindent

{\em M. Antunes}
\vspace{0.2cm}

\noindent

Departamento de Ciências Exatas,

Universidade Federal Fluminense - UFF

Avenida dos trabalhadores, 420, Vila Santa Cecília

Volta Redonda - RJ, Brasil.

\vspace{0.2cm}

\email{mayaraantunes@id.uff.br}

\vspace{1.0cm}
{\em B. Carvalho}
\vspace{0.2cm}

\noindent

Dipartimento di Matematica,

Università degli Studi di Roma Tor Vergata

Via Cracovia n.50 - 00133

Roma - RM, Italy
\vspace{0.2cm}

\email{carvalho@mat.uniroma2.it}

\vspace{1.0cm}
{\em W. Cordeiro}
\vspace{0.2cm}

\noindent

Instituto de Matem\'atica

Universidade Federal do Rio de Janeiro

Av. Athos da Silveira Ramos, 149 - Cidade Universitária 

Rio de Janeiro - RJ, Brazil

\vspace{0.2cm}

\email{welingtonscordeiro@gmail.com }

\vspace{1.0cm}
{\em J. Cueto}
\vspace{0.2cm}

\noindent

Departamento de Matem\'atica,

Universidade Federal de Minas Gerais - UFMG

Av. Ant\^onio Carlos, 6627 - Campus Pampulha

Belo Horizonte - MG, Brazil.

\vspace{0.2cm}

\email{joscueto@hotmail.com }

\end{document}